\documentclass[12pt]{article}

\usepackage{amsmath}
\usepackage{amsthm}
\usepackage{graphicx}
\usepackage{amsfonts}
\usepackage{fullpage}
\usepackage{amssymb}
\usepackage{cite}
\usepackage{hyperref}
\usepackage{harmony}

\title{On The Euclidean Algorithm: Rhythm Without Recursion}
	\author{Thomas Morrill\\
	        Trine University\\
	        Angola, Indiana, USA\\
	        \href{mailto:morrillt@trine.edu}{\nolinkurl{morrillt@trine.edu}}
	        }

\newtheorem{cor}{Corollary}
\newtheorem{thm}{Theorem}
\newtheorem{lemma}{Lemma}

\newcommand{\ZZ}{\mathbb{Z}}

\newcommand{\note}{\Vier \ }
\newcommand{\rest}{\ViPa \ }

\begin{document}

\maketitle

\begin{abstract}
  A modified form of Euclid's algorithm has gained popularity among musical composers following Toussaint's 2005 survey of so-called Euclidean rhythms in world music.
  We offer a method to easily calculate Euclid's algorthim by hand as a modification of Bresenham's line-drawing algorithm.
  Notably, this modified algorithm is a non-recursive matrix construction, using only modular arithmetic and combinatorics.
  This construction does not outperform the traditional divide-with-remainder method;
  it is presented for combinatorial interest and ease of hand computation.
\end{abstract}

\section{Introduction}

In 2005, Godfried Toussaint described a method of producing musical rhythms via the Euclidean algorithm.
These are commonly known as \emph{Euclidean rhythms}.
Euclid's well-known algorithm calculates the greatest common divisor of two integers $k$ and $N$ using integer division and recursion.
Toussaint's application generates a rhythm consisting of $k$ notes arranged in a pattern of $N$ beats.
Euclid's algorithm has been greatly influential in number theory, and more broadly the study of integral domains.
So, it may come as a surprise that Euclid's work has a practical application to musical composition.

Arguing for the aesthetic value of Euclidean rhythms is outside the scope of this work.
However, it is worth noting that Toussaint was able to catalogue a large number of Euclidean rhythms where they already exist in music traditions across many disparate world cultures.
A few examples of extant rhythms that may be reproduced by the Euclidean algorithm include traditional rhythms from Cuba, Persia, South Africa, and Turkey \cite{Toussaint-1}.

In the fifteen years following Toussaint's publication, the concept of Euclidean rhythms has become popular among music theorists and composers.
It is particularly popular among electronic and experimental composers;
many companies now produce Euclidean rhythm generators as commercial hardware or software, or incorporate these capabilities into other multi-function products.

For our purposes, a \emph{binary rhythm}, or more succinctly a \emph{rhythm}, is a finite sequence whose \emph{symbols} come from a set of two elements.
The symbols are arbitrary, but assumed to be opposites;
they may be taken from sets such as $\{0,1\}$, $\{<, >\}$, or $\{\textsc{Off}, \textsc{On}\}$.
We choose the symbol set $\{\ViPa, \Vier \}$, in keeping with the context of Toussaint's work.
Each rhythm has a length, and each symbol occuring in the rhythm is either a \emph{note} (\Vier) or a \emph{rest} (\ViPa).
Rhythms are treated cyclically;
for example \note \note \note \rest \rest is equivalent to \note \rest  \rest \note \note \cite{Toussaint-3}.

Toussaint utilizes Eric Bjorklund's algorithm as one method for generating Euclidean rhythms.
This algorithm originates in the theory of discrete-time control systems \cite{Bjorklund-1}.
Bjorklund's algorithm produces a rhythm of length $N$ which contains $k$ \textsc{On}s spaced as evenly as possible among $(N-k)$ \textsc{Off}s \cite{Bjorklund-2}.
This is achieved by recursively concatenating Euclidean rhythms of shorter length according to back-substitution.
At a conceptual level, the procedure resembles Euclid's recursive divide-with-remainder approach, using subsequences and symbol replacement in place of integer division and back-substitution  \cite{Toussaint-3}.
These two algorithms are in fact equivalent.

\begin{figure}
\[
\begin{bmatrix}
-1 & 0 & 1 & 2 & 3 &  4 &  5 &  6 \\
-3 & 0 & 3 & 6 & 9 & 12 & 15 & 18 \\
 4 & 0 & 3 & 6 & 2 &  5 &  1 &  4
\end{bmatrix}
\]
\caption{The Euclidean array $E(3, 7)$. Three descents occur along the bottom row.}
\label{three-seven-array}
\end{figure}

Other applications of Euclid's algorithm outside number theory include calculation of leap years, and in particular, Bresenham's line drawing algorithm \cite{Bresenham}.
The latter was developed for displaying line segments on digital displays.
We propose a variant of Bresenham's algorithm which constructs what we call the \emph{Euclidean array} $E(k,N)$.
This array contains $\gcd(k,N)$ as a minimal element.
In practical terms, $E(k,N)$ records the intermediate results of Bresenham's subtract-and-carry loop.
It is intended for ease of hand computation, using a written array in place of dynamic computer memory, and combinatorics in place of carries.

Of course, Euclid's algorithm has another important application beyond finding $\gcd(k, N)$.
The extended Euclidean algorithm also calculates solutions to the Diophantine equation
\begin{align} \label{Diophantine}
  ak + bN = \gcd(k, N)
\end{align}
via back-substitution.
Our non-recursive algorithm solves \eqref{Diophantine} in three steps:
\begin{enumerate}
  \item Construct $E(k, N)$.
  \item Locate the least positive entry on the third row of $E(k, N)$.
  \item Calculate $a$ and $b$ using the explicit formulas in Theorem \ref{extended}.
\end{enumerate}

Two methods to construct the arrays $E(k, N)$ appear in Section \ref{euclidean-arrays}.
We sketch the proof of our algorithm's correctness in Section \ref{elementary}.
Section \ref{descent} presents a method for determining Euclidean rhythms from the $E(k,N)$.
This method involves identifying the descents of the array, treated as a permutation.
Finally, Section \ref{end} ends on some closing remarks.

\section{Euclidean Arrays}
\label{euclidean-arrays}

Given $0 \leq k \leq N$, the \emph{Euclidean array} $E(k,N)$ is a $3 \times (N+1)$ array constructed as follows.
Row 1 is the arithmetic progression $-1, 0, 1, \ldots, (N-1)$ of length $N+1$.
It is easily written down by hand.
The entries of row 2 are each equal to $k$ times the corresponding entry from row 1.
This row may be constructed either via multiplication of the entries on row 1 by $k$, or by repeated addition of $k$ to $-k$.
The entries of row 3 are each equal to the corresponding entry from row 2, reduced modulo $N$ to its representative in the set $\{0, 1, \ldots, (N-1)\}$.
When calculating by hand, one can take advantage of the fact that exactly one of the numbers $(m+k)$ and $(m - (N-k))$ appear in the residue set $\{0, 1, \ldots, (N-1)\}$ if $m$ is already such a residue.

We refer to the rows of $E(k, N)$ as the \emph{index row}, \emph{integer row}, and \emph{residue row}, respectively.
An example is given in Figure \ref{three-seven-array}.

Some explanation regarding this construction is necessary.
We choose to start indexing at $-1$ so that the corresponding Euclidean rhythm begins with a note.
Also, we deliberately avoid treating the indices cyclically.
While doing so would allow the omission of one column, we consider it to be an oversimplification.
Instead, we intend Euclidean arrays to be marked with the symbols $>$ and $<$ as in Figure \ref{marked} for identifying descents during hand computation.
This is further discussed in Section \ref{descent}.

\begin{figure}
\[
\begin{bmatrix}
-1 & & 0 & & 1 & & 2 & & 3 \\
-2 & & 0 & & 2 & & 4 & & 6 \\
 2 &>& 0 &<& 2 &>& 0 &<& 2
\end{bmatrix}
\]
\caption{The Euclidean array $E(2, 4)$. Its ascents and descents are marked along the bottom row.}
\label{marked}
\end{figure}

\section{Elementary Results} \label{elementary}

Consider a Euclidean array $E(k,N)$ with $0 \leq k \leq N$.
The following results are easily proven using elementary number theory\footnote{Lemma \ref{exercise} is vaculously true when $k = 0, N$.}.

\begin{lemma} \label{exercise}
	The least positive entry on the residue row is equal to $\gcd(k, N)$.
	The greatest positive entry on the residue row is equal to $N - \gcd(k, N)$.
\end{lemma}

\begin{thm} \label{extended}
	Given $E(k, N)$, if $[a_1, a_2, a_3]^T$ is a column vector such that $a_3 = \gcd(k,N)$, then $a_{1} k + b N = \gcd(k, N)$, where
	\[
	  b = \frac{a_2 - a_3}{N}.
	\]
\end{thm}

This theorem occasionally gives solutions to the Diophantine equation \eqref{Diophantine} that are true, but inelegant.
For example, the array in Figure \ref{three-seven-array} produces the equation
\[
	5 \cdot 3 - 2 \cdot 7 = 1,
\]
rather than the more obvious
\[
	1 \cdot 7 - 2 \cdot 3 = 1.
\]
An alternative solution can sometimes overcome this blemish.

\begin{cor} \label{complement}
	Given $E(k, N)$, if $[b_1, b_2, b_3]^T$ is a column vector such that $b_3 = N - \gcd(k,N)$, then $aN - b_{1} k = \gcd(k, N)$, where
	\[
	  a = 1 - \frac{b_2 - b_3}{N}.
	\]
\end{cor}

We pause at the threshold to Section \ref{descent} to optimize our construction.
To construct a \emph{reduced} Euclidean array $\hat{E}(k,N)$, proceed by completing columns, not rows.
Terminate the array after the first repeated entry on the residue row, regardless of the number of columns.
An example is given in Figure \ref{four-six-array}.

There is a further optimization if one is only concenered with $gcd(k, N)$ and not the corresponding Euclidean rhythm.
If either of the residues $1$ or $(N-1)$ occur, as in Figure \ref{three-seven-array}, then the construction may be immediately interrupted in favor of applying Lemma \ref{exercise}, Theorem \ref{extended}, or Corollary \ref{complement}, as in both cases $gcd(k,N) = 1$.

\begin{figure}
\[
\begin{bmatrix}
-1 & 0 & 1 & 2 \\
-4 & 0 & 4 & 8 \\
 2 & 0 & 4 & 2
\end{bmatrix}
\]
\caption{The reduced Euclidean array $\hat{E}(4, 6)$.}
\label{four-six-array}
\end{figure}

\section{Determining Euclidean Rhythms}
\label{descent}

Consider the residue row $[n_{-1}, n_0, \ldots, n_{N-1}]$ of a Euclidean array $E(k,N)$.
Each pair of adjacent entries $(n_{i}, n_{i+1})$ is called an \emph{ascent} if $n_{i} < n_{i+1}$, and a \emph{descent} if $n_{i} > n_{i+1}$.
Ascents and descents appear in the study of permutations, notably in the definition of the Eulerian numbers \cite{Knuth}.

We construct the \emph{Euclidean rhythm} corresponding to $E(k, N)$ as follows.
The $i$th entry of the Euclidean rhythm is \note if the pair $(n_{i-1}, n_{i})$ is a descent, and \rest otherwise.
The Euclidean rhythm corresponding to Figure \ref{three-seven-array} is \note \rest \rest \note \rest \note \rest.
The residue row of $E(N,N)$ is identical to that of $E(0,N)$, so we define its corresponding Euclidean rhythm to be \note \note \ldots \note.

This procedure may also be performed on the reduced array $\hat{E}(k, N)$, with the observation that Euclidean rhythms are periodic with minimal period length $N/\gcd(k,N)$.
The Euclidean rhythm corresponding to Figure \ref{four-six-array} is \note \rest \note \note \rest \note.

We now prove the following lemma.

\begin{lemma}
	The Euclidean rhythm obtained from $E(k, N)$ contains exactly $k$ notes.
\end{lemma}

\begin{proof}
Consider a line segment in the Cartesian plane with endpoints $(-1, -k)$ and $(N-1, Nk - k)$.
This segment interpolates the first two rows of the Euclidean array $E(k,n)$, with the index row giving the $x$-coordinates, and the integer row giving the $y$-coordinates.

A descent appears on the residue row of $E(k,N)$ between columns $i-1$ and $i$ if and only if the line segment intersects a horizontal line of the form $y = Nt$  with $t \in \ZZ$ on the $x$-interval $(i-1, i]$.
The line segment crosses $k$ such lines, whose formulas are $y = 0, N, 2N, \ldots, (k-1)N$.
\end{proof}

We close the section with two smaller results.
The first allows computation of $\gcd(k, N)$ from a Euclidean rhythm without reference to $E(k, N)$.

\begin{cor}
	Given a Euclidean rhythm $R$ of length $N$ which contains $k$ notes, $\gcd(k, N)$ is equal to the number of occurences of the minimal period of $R$.
\end{cor}

For example, the Euclidean rhythm \note \rest \note \note \note \rest \note \note has two occurrences of its minimal period \note \rest \note \note.
This demonstrates that $\gcd(6, 8) = 2$.
Our final result shows that the set of all Euclidean rhythms is invariant under the involution $\Vier \leftrightarrow \rest$.

\begin{lemma}
	Constructing the rhythm corresponding to $E(k, N)$, but instead assigning notes to ascents, and rests to non-ascents, is equivalent to constructing the rhythm corresponding to $E(N-k, N)$ via the ordinary method.
\end{lemma}

In other words, Euclidean rhythms distribute their rests in the same manner as their notes.
We refer the curious reader to Bjorklund's work for the proof of Euclidean rhythms' equal distribution properties \cite{Bjorklund-2}.
Many other theoretical properties of the Euclidean rhythms have been described by Gomez-Martin, Taslakian, and Toussaint \cite{Toussaint-2}.

\section{Discussion} \label{end}
We have now considered applications of Euclid's algorithm from outside mathematics, and returned with a some new mementos.
This work demonstrates a connection between the Euclidean algorithm, musical rhythms, sequences of residues, and combinatorics.
We hope to entertain students of number theory and discrete mathematics at all levels.

One imagines a novice asking ``How many times do I repeat the division step in Euclid's algorithm?'' and not being satisfied with the response ``About $\log N$ times, you'll know when you're done.''
I find our alternate repsonse, ``Calculate $N$ residues, then look for the smallest'' to be interesting for its definiteness, despite being a worse answer in terms of computational complexity.

Construction of the Euclidean arrays avoids the common pitfalls of Euclid's algorithm -- keeping track of which integers are the quotients, divisors, and remainders, and arranging them correctly during back-substitution.
Perhaps the $E(k,N)$ will find a home in maths pedagogy.
Further, a musical composer who wishes to apply Euclidean rhythms may find it easier to work with the array $E(k, N)$ by hand compared to Bjorklund's algorithm under the same circumstance.

It remains an open and subjective question whether identifying descents in other integer sequences would produce musically interesting rhythms.



\end{document}